\documentclass[12pt,reqno]{article}
\usepackage{amssymb}
\usepackage{graphicx}
\usepackage{amsmath}
\usepackage[utf8]{inputenc}
\usepackage{textcomp}
\usepackage{amssymb,indentfirst}
\usepackage{times}
\usepackage[T1]{fontenc}
\usepackage{tikz}
\usepackage[toc,page]{appendix}
\usepackage{lipsum}

\newtheorem{thm}{Theorem}[section]
\newtheorem{cor}[thm]{Corollary}
\newtheorem{lem}[thm]{Lemma}
\newtheorem{prop}[thm]{Proposition}
\newtheorem{defn}[thm]{Definition}
\newtheorem{rmk}[thm]{Remark}

\newcommand{\Ad}{\operatorname{Ad}}
\newcommand{\tr}{\operatorname{tr}}

\def\Z{{\mathbb Z}}

\def\R{{\mathbb R}}
\def\C{{\mathbb C}}

\def\bb{\begin}
\def\bc{\begin{center}}       \def\ec{\end{center}}
\def\be{\begin{equation}}     \def\ee{\end{equation}}
\def\ba{\begin{array}}        \def\ea{\end{array}}
\def\bea{\begin{eqnarray}}    \def\eea{\end{eqnarray}}
\def\beaa{\begin{eqnarray*}}  \def\eeaa{\end{eqnarray*}}
\def\bma{\begin{pmatrix}}     \def\ema{\end{pmatrix}}
\def\bsl{\begin{split}}       \def\esl{\end{split}}

\def\hh{\!\!\!\!}             
\def\EQ{\hh & = & \hh}

               \def\om{\omega}
              \def\vp{\varphi}
               \def\th{\theta}
           \def\ga{\gamma}

               \def\Ga{\Gamma}
               \def\ro{\rho}

\def\M{\mathcal{M}}           \def\tro{\tilde{\rho}}

\def\nn{\nonumber}
\def\oo{\infty}                              
                              
\def\q{\quad}                                \def\qq{\qquad}
\def\f{\frac}                                
\def\z{\left}                                \def\y{\right}
                              
\def\ol{\overline}
\def\ul{\underline}

\def\A{{\mathcal A}}

\def\mcc{{\mathcal C}}
\def\D{{\mathcal D}}
\def\E{{\mathcal E}}

\def\M{{\mathcal M}}

\def\mcr{{\underline{R} }}

\def\bfr{{\bf r}}

\def\rd{\,{\rm d}}
\def\dt{\,{\rm d}t}
\def\ds{\,{\rm d}s}

\def\dx{\,{\rm d}x}
\def\dy{\,{\rm d}y}

\def\ifl{\iffalse}
\def\d{\cdot}

\def\oo{\infty}
\def\f{\frac}

\def\z{\left}
\def\y{\right}
\def\q{\quad}
\def\qq{\qquad}

\def\andq{\quad \mbox{ and } \quad}
\def\qqf{\quad \forall}
\def\lb{\label}
\def\nn{\nonumber}
\def\x#1{{\rm (\ref{#1})}}
\def\Proof{\noindent{\bf Proof} \quad}

\def\qed{\hfill $\Box$ \smallskip}

\def\vr{{\gamma}}

\def\tr{\mathrm{tr}}

\def\ysb#1{{\color{blue} #1}}

\begin{document}

\title{Geometric Invariants in Bicycle Dynamics via the Rotation Number Approach}
\author{Diantong LI, Qiaoling WEI, Meirong ZHANG\footnote{M.Z. was supported by the National Natural Science Foundation of China (Grant no. 11790273)}, Zhe ZHOU\footnote{Z.\ Z.\ was supported in part by the National Natural Science Foundation of China (Grants No. 12271509, 12090010, 12090014).}}

\date{}
\maketitle

\begin{abstract}

We study planar bicycle dynamics via the rotation number function \(\rho_\Gamma(R)\) associated with a closed front track \(\Gamma\) and bicycle length \(R\).
We prove that mode-locking \emph{plateaus} occur only at \emph{integer} rotation numbers and that \(\rho_\Gamma\) is real-analytic in \(R\) off resonance.
From \(\rho_\Gamma\) we introduce two new geometric invariants: the \emph{critical B-length} \(\mcr(\Gamma)\) (right end of the first plateau) and the \emph{turning B-length} \(\overline{R}(\Gamma)\) (left end of the maximal monotone interval).
For a star-shaped curve \(\Gamma\), these invariants coincide, yielding a sharp transition of the bicycle monodromy: hyperbolic for \(R<\mcr(\Gamma)\) and elliptic for \(R>\mcr(\Gamma)\).
The proofs combine projectivized \(SU(1,1)\) dynamics with Riccati equations and rotation-number theory.


\end{abstract}

\bigskip




\section{Introduction}
\setcounter{equation}{0} \lb{first}

The bicycle dynamics (BCD, for short) is a simple model with rich geometry and numerous links to classical mechanics and differential equations. In dimension $d$, a bicycle has a steerable front wheel and a rear wheel constrained to follow the instantaneous direction of the segment of fixed length $R$ connecting the rear and front contact points. Given the motion of the front tire by a parameterized curve $\Gamma$, called {\it front track}, the motion of the rear tire, called {\it rear track}, is any parameterized curve  $\gamma$ that satisfies:
\begin{align}
    &\|\Gamma(t)-\gamma(t)\|= R, \label{BC1}\\
   & \text{the tangent vector}\, \gamma'(t) \,\text{is parallel to}\, \gamma(t)-\Gamma(t).\label{BC2}
\end{align}


 In the plane, BCD is closely related to the geometry of the Prytz (hatchet) planimeter, and its time-$2\pi$ map on the circle is a M\"obius transformation, where the dynamics can be  divided into three classes: hyperbolic, parabolic or elliptic according to the number of fixed points of the M\"obius transformation on $S^1$, i.e. two, one, or zero respectively.
 See early geometric work of R.~Foote~\cite{F98}, the comprehensive treatment by Foote--Levi--Tabachnikov~\cite{FLS13}, and further developments by M.~Levi~\cite{Le17}, Bor--Levi--Perline--Tabachnikov~\cite{BLPT20}, Levi--Tabachnikov~\cite{LT09}, Bor--Levi~\cite{BL20}, and Bor--Jackman--Tabachnikov~\cite{BJT23}.

In this paper we focus on the planar case where the front wheel travels along a closed curve $\Gamma$. We introduce the \emph{rotation-number function} $\rho_\Gamma(R)$ of the bicycle length $R$, defined via the bicycle equation. This function is independent of the parametrizations of $\Gamma$ and only depends on the geometry of the curve. It encapsulates the asymptotic rear-wheel dynamics and, at the same time, reflects classical geometric data of $\Gamma$, including perimeter, enclosed (signed) area, rotation index, and curvature radii.

From a dynamical-systems viewpoint, the steering-angle equation (see Eq. {\bf (\ref{normal})}) may be regarded as a continuous analogue of families of circle maps that exhibit \emph{mode-locking} and \emph{Arnold tongues}; in our setting, we prove that plateaus of $\rho_\Gamma(R)$ can occur only at \emph{integer} rotation numbers, and that $\rho_\Gamma(R)$ is real-analytic off those resonance values. These phenomena connect BCD to classical rotation-number theory and Hill/Riccati-type equations, topics that also appear in~\cite{BL20,LT09}.

A central theme of our work is that $\rho_\Gamma(R)$ naturally singles out two geometric invariants of $\Gamma$. The first one is the \emph{critical B-length} $\mcr(\Gamma)$: it is the right end of the first plateau  of $\rho_\Gamma(R)$. The second is the \emph{turning B-length} $\overline{R}(\Gamma)$ : it is the left endpoint of the maximal interval on which $\rho_\Gamma(R)$ is strictly monotone, used to characterize the asymptotic behavior of $\rho(R)$.
The behavior of the function $\rho_{\Gamma}(R)$ , as well as the two invariants play an important role in the classification of the types of monodromy maps of BCD.  Estimates of $\mcr(\Ga)$ are given by inequalities using classical geometric invariants. In particular, an isoperimetric-type inequality involving $\mcr(\Ga)$ exists. See inequality (\ref{isoine}).

Finally,  we study a conjecture raised in \cite{WZ22} : for strictly convex closed curve $\Gamma$, the two geometric invariants introduced above coincide. That means, the associated rotation number function has only one plateau and is strictly decreasing elsewhere. It implies that the invariants $\mcr(\Ga)=\overline{R}(\Ga)$ serve as a criterion to distinguish the types of the M\"obius transformation of BCD, i.e., hyperbolic when $R<\mcr(\Gamma)$, elliptic when $R>\mcr(\Gamma)$. This  addresses the conjecture of monodromy types discussed recently by Bor--Hern\'andez-Lamoneda--
Tabachnikov~\cite{BHT24}. We prove both conjectures for a larger class of curves: star-shaped curves.

Historically, a preliminary version of this project appeared as the preprint~\cite{WZ22} by the second and third author. In \cite{WZ22}, being unaware of the bicycle track literature at the time, the problem was formulated in the language of “shadowing” : suppose that one person, called an escaper, is escaping along a parametrized curve $\Gamma (t)$, called escaping curve, where $t$ is the time. Another person, called a shadower, is shadowing the escaper using the simplest strategy {\it by staring at the escaper and keeping the initial distance from the escaper at all times}, his trace gives a shadowing curve $\gamma(t)$. It is apparent that $\Gamma$ and $\gamma$ are exactly the front track and rear track in BCD respectively.
The present paper recasts the framework in the standard BCD terminology, situates it within the classical works cited above, and strengthens the main conclusions:
\begin{itemize}
    \item Mode-locking (or plateaus) occurs only at integer levels and  $\rho_{\Gamma}$ is real analytic away from resonance ({\bf Theorem \ref{anal}});
    \item For star-shaped curve $\Gamma$, $\rho_{\Gamma}$ has a unique plateau, i.e. $\mcr(\Gamma)=\overline{R}(\Gamma)$ ({\bf Theorem \ref{main2}}) and $\mcr(\Gamma)$ serves as the unique transition length of the monodromy map from hyperbolic to elliptic ({\bf Theorem \ref{trans}}). 
\end{itemize}

\section{Rotation number and monodromy maps}
\setcounter{equation}{0} \lb{original}

\subsection{BCD equations}
From now on, we are concentrating on the BCD on the plane $\R^2$. Moreover, suppose that the front tire track $\Ga(t)$ is a regular $C^2$ closed curve\footnote{ The regularity can be indeed generalized to be piecewise $C^2$. However, for simplicity, we only consider the $C^2$ case in this paper.}:
    \be\label{Ga}
    \Ga: \q\Ga(t)=(\xi(t),\eta(t))\in C^2(\R,\R^2),\quad (\xi(t+L),\eta(t+L))\equiv (\xi(t),\eta(t)).
    \ee
with $\|\Gamma'(t)\|>0$ and some period $L>0$.
 The rear track in BCD is determined by solutions of first order ODE,  we  summarize various forms of the equations, which appeared similarly, for example, in \cite{FLS13,BLPT20}.

    \bb{lem} \lb{BCD0}
The rear track $\ga(t)$ is determined by the following first-order system of ODEs in $\R^d$
    \be \lb{SE}
    \ga'(t)  =\f{\Ga'(t)\d (\ga-\Ga(t))}{\|\ga-\Ga(t)\|^2}(\ga-\Ga(t)), \qq t\in \R.
    \ee

    \end{lem}

\Proof It is easy to see geometrically that the requirement \x{BC1} and \x{BC2} are fulfilled if and only if the tangent vector of the rear track equals to the projection of the tangent vector  of the front track $\Gamma'(t)$ along the direction $\gamma(t)-\Gamma(t)$.

We call the initial distance
\[R:=\|\gamma(0)-\Gamma(0)\|\] the {\it bicycle length} in the sequel.
When $R>0$ is fixed, Eq. \x{SE} is an ODE on the circle $S^1= \R/L \Z$. By denoting
    \[
    \bfr(t) = (\cos \th(t), \sin \th(t))\in S^1, \qq \vr(t)= \Ga(t)+ R\bfr(t),
    \]
one can deduce from Eq. \x{SE} that the equation for $\th=\th(t)$ is
    \be \lb{RSE}
    \th'   = \f{1}{R}(\xi'(t)\sin \th -\eta'(t)\cos \th):=F_R(t,\th).
    \ee
Solutions of Eq. \x{RSE} satisfying $\th(0)=\th_0\in\R$ are denoted by
$\th(t,\th_0)$ or by $\th_R(t,\th_0)$ when $R$ is emphasized.

Another useful version of BCD equation is to consider the equation of the {\it steering angle} $\alpha(t)$, that is, the angle from the tangent vector $\Gamma'(t)$ to the direction $\gamma(t)-\Gamma(t)$. For this purpose, write $\Ga'(t)$ in polar coordinates as
    \be \lb{ec23}
    \Ga'(t) \equiv B(t) \z( \cos \psi(t), \sin \psi(t)\y)= B(t) e^{i \psi(t)},
    \ee
where $B(t)=\|\Ga'(t)\|>0$ is a continuous $2\pi$-periodic function, and $\psi(t)$ is a continuous function which is uniquely determined by choosing
    \(
    \psi(0) \in[-\pi,\pi).
    \)

Then the steering angle is $\alpha=\th-\psi$.
\begin{lem} If $\Gamma$ is a regular $C^2$  closed curve, the steering angle $\alpha$ satisfies the following ODE:
\begin{equation}\label{eq:angle}
    \alpha'(t)=B(t)(\frac{1}{R}\sin\alpha-\kappa(t))=:f_R(t,\alpha),
\end{equation}
 where $\kappa(t)$ is the signed curvature function of the front track $\Gamma$.
\end{lem}

\begin{proof}
    With the polar form \x{ec23}, the BCD equation \x{RSE} is written as
    \[\th'=\frac{1}{R}B(t)\sin(\th-\psi),\]
    Plug in $\alpha=\th-\psi$, we arrive at (\ref{eq:angle}) using the geometric fact that
    \[\kappa(t)=\frac{\det(\Ga'(t),\Ga''(t))}{\|\Ga'(t)\|^3}=\frac{\psi'(t)}{B(t)}.\]

\end{proof}

\subsection {Rotation number function $\rho_{\Gamma}(R)$}
By virtue of the periodicity of Eq. \x{RSE}, one knows that the monodromy maps $\th_0 \mapsto \th(t,\th_0)$ satisfies
    \be \lb{P}
    \th(t,\th_0+2k\pi) = \th(t,\th_0) + 2k\pi \qqf t\in \R, \ k \in \Z.
    \ee
Moreover, as Eq. (\ref{RSE}) is a time-periodic ODE on the circle $S^1$, it is well-known that it admits a rotation number \cite{A83, H80, KH95}:

   \be \lb{rhoF}
    \ro(R)=\rho_{\Gamma}(R):=\frac{L}{2\pi}\lim_{t\to \infty}\frac{\th_R(t,\th_0)-\th_0}{t}=\lim_{n\to\infty}\frac{1}{2\pi}\frac{\mathcal{P}_R^n(\theta_0)-\theta_0}{n}.
    \ee
  where $\mathcal{P}_R:\theta_0\mapsto \theta(L,\theta_0)$ denotes the Poincar\'e map.  See for example Theorem 2.1 in \cite{H80}.

    \bb{defn} \lb{rn}
{\rm Given a closed curve $\Ga$ as in \x{Ga}, we use $\ro(R)=\ro_\Ga(R)$ to denote the rotation number of Eq. \x{RSE}. It is called the rotation number function of $\Ga$.}
    \end{defn}

We remark that $\ro(R)$ does not need to be taken modulo $\Z$, and is a continuous function of $R\in(0,+\oo)$.

It will also be convenient to consider  the rotation number defined by the equation \x{eq:angle} :
\[\tilde{\rho}(R)=\tilde{\rho}_{\Gamma}(R):=\lim_{t\to \infty}\frac{\alpha_R(t,\alpha_0)-\alpha_0}{t}.\]

Due to the periodicity of $\Ga'(t)$, one has
    \be \lb{Avp}
    \psi(t+L) \equiv 2 \om \pi+ \psi(t),
    \ee
where
    \be \lb{omega}
    \om= \om(\Ga):=\frac{1}{2\pi}\int_0^{L}\psi'(t)\dt
    \ee
is an integer.
This number $\om$ is the {\it winding number} of the derivative closed curve $\Ga'$ (not $\Ga$ itself). It is also called the {\it rotation index} of $\Ga$ (\cite{dC76})
 which measures the complete turns given by the tangent vector field along the closed curve $\Ga$ around the origin. A non-trivial Jordan curve (a simple closed curve) $\Ga$ has rotation index $\pm 1$.

It follows directly that
    \[\tilde{\rho}(R)=\rho(R)-\omega.\]

The following lemma shows that the rotation number function $\rho_{\Gamma}(R)$ only depends on the geometry of the curve.
\begin{lem} \lb{inv}
    For any fixed length $R>0$, the rotation number $\ro_{\Ga}(R)$ is independent of the parametrization of $\Ga$.
\end{lem}

\Proof
Let $\Ga$ be as in \x{Ga} and \x{ec23}, which are parametrized using $t$. The arc-length parameter $\bar{s}$ is given by
    \[
    \bar{s}=\int_0^t \|\Ga'(t)\|\dt'=:\beta(t).
    \]
It follows that
    \[
    \beta(t+L)\equiv \beta(t)+\int_0^{L} \|\Ga'(t)\|\dt=\beta(t)+\ell,
    \]
where $\ell=\ell(\Ga)$ is the  perimeter of $\Ga$. Then
    \[
    \beta^{-1}(\bar{s}+\ell)\equiv \beta^{-1}(\bar{s})+L.
    \]
The arc-length parametrization of $\Ga$ is $
    \Ga_*(s):=\Ga(\beta^{-1}(s)),
$
which is $\ell$-periodic in $s$.

Given a bicycle length $R$, if $\th=\th(t)$ is a solution of the  BCD equation \x{RSE} to $\Ga(t)$,
it is direct to check that $\th_*(s):=\th(\beta^{-1}(s))$ is a solution to $\Ga_*(s)$, i.e.
    \[
   \th_*'(s) = \f{1}{R}\Ga'_*(s) \d (\sin \th_*(s), -\cos \th_*(s)),
    \]
Thus
    \[
    \ro_{\Ga_*}=\frac{\ell}{2\pi}\lim_{s\to \infty}\frac{\th_*(s)}{s}=\frac{\ell}{2\pi}\lim_{s\to\infty}\frac{\th(\beta^{-1}(s))}{\beta^{-1}(s)}\d\lim_{s\to\infty}\frac{\beta^{-1}(s)}{s}=\frac{L}{2\pi}\lim_{t\to\infty}\frac{\th(t)}{t}=\ro_{\Ga}.
    \]

\begin{cor}\lb{arnold}
 Given any $c>0$, let $\Gamma_{c}:=c\Gamma$. Then $\ro_{\Gamma_{c}}(R)= \ro_{\Gamma}(c^{-1} R)$.

\end{cor}
\begin{proof}  The statement follows directly from the BCD Eq. \x{RSE}.\qed
\end{proof}

\begin{lem} \label{upbound}
For any regular $C^2$ closed  curve $\Ga$ as in \x{Ga}, the rotation number function $\ro_\Ga(R)$ satisfies
    \be \lb{rhoe1}
    |\ro_\Ga(R)|\leq \frac{\ell(\Ga)}{2\pi R}, \qquad \forall R\in (0,\infty).
    \ee
 In particular, one has
    \(
    \lim_{R\to+\oo} \ro_\Ga(R)=0.
    \)
    \end{lem}

\Proof  One has from Eq. \x{RSE} that
    \[
    |\th'(t)| \le \f{1}{R} \|(\xi'(t),-\eta'(t))\| =\f{1}{R} {\|\Ga'(t)\|}.
    \]
Hence
    \[
    |\th(t)|\le |\th(0)| +  \f{1}{R} \int_0^t \|\Ga'(s)\| \ds\qqf t\ge 0,
    \]
and
   \[
   |\ro(R)| \le \lim_{t\to\oo} \f{L}{2\pi t}\z(|\th(0)| +  \f{1}{R} \int_0^t \|\Ga'(s)\| \ds\y) = \f{1}{2\pi R} \int_0^{L} \|\Ga'(s)\|\ds = \f{\ell(\Gamma)}{2\pi R}.
    \]
\qed

In view of Corollary \ref{arnold}, one can always study $\rho_{\Ga}$ for  $\Ga$ in arc-length parametrization. With this normalization, the BCD Eq. \x{eq:angle} for the steering angle $\alpha$  writes as
\be\lb{normal}\alpha'=\frac{\sin \alpha}{R}-\kappa(s)\ee

This equation can be viewed as a continuous counterpart of the discrete circle maps
\[\alpha_{i+1}=\alpha_i+\Omega+ K\sin \alpha_i\]
where Arnold tongues and mode-locking appear.

More precisely, a family of circle maps with parameter $R$ is mode-locked at level $a$ if there is a nontrivial interval $I$ of parameters with rotation number $\rho(R)=a$ for all $R\in I$. Such an interval is a plateau (mode-locking interval), and we say that $\rho$ is resonant at $a$. In our setting, the bicycle equations gives rise to a family of circle maps defined by the Poincar\'e map $\mathcal{P}_R$.
Thus it is interesting to discuss the mode-locking phenomena for $\rho_{\Gamma}(R)$.

The next example shows that the rotation number of a circle has a unique plateau.

  \bb{lem} \lb{rhoc}
The rotation number function of the unit circle $\mcc$ is given by
    \[
    \ro_\mcc(R)=
    \z\{ \ba{ll}
    1 & \mbox{ for } R\in(0,1],\\
    1-\f{\sqrt{R^2-1}}{R}
    & \mbox{ for } R\in[1,+\oo).
    \ea\y.
    \]
    \end{lem}
\begin{proof}
    Consider the BCD Eq. \x{eq:angle}:
    \[\alpha'=\frac{\sin\alpha}{R}-1,\]

If $R\leq 1$, $\alpha'=0$ has solutions, i.e. fixed points exist, thus $\tilde{\rho}(R)=0$.

If $R>1$, we can find this by calculating the time $T_\alpha$ it takes for $\alpha$ to change by one full period, $\Delta\alpha = -2\pi$.
$$T_\alpha = \int_0^{T_\alpha} dt = \int_{\alpha_0}^{\alpha_0 - 2\pi} \frac{d\alpha}{\alpha'} = \int_0^{2\pi} \frac{R}{R - \sin\alpha} d\alpha=\frac{2\pi R}{\sqrt{R^2 - 1}}.$$
The rotation number $\tilde{\rho}_{\mcc}(R)$ is the total change in angle divided by the time taken:
$$\tilde{\rho}_{\mcc}(R) = \frac{\Delta\alpha}{T_\alpha}  = -\frac{\sqrt{R^2 - 1}}{R}.$$
\end{proof}
\qed

\subsection{Monodromy maps}
In the remainder of the paper, $\Gamma$ is a regular $C^2$ closed curve taken to be 
$2\pi$ periodically parametrized, unless explicitly stated otherwise.

It is classical to introduce the monodromy maps of BCD on $S^1$:
 \be
    \M_R^t : S^1\to S^1,\quad z=e^{i\theta_0} \mapsto z(t)=e^{i\theta_R(t,\theta_0)}
    \ee
where $\th_R(t,\th_0)$ is a solution of \x{RSE}. It is known that $\M_R^t$ is a M\"obius map, see \cite{F98}. For completeness, we give a self-contained proof in the following.

   \bb{thm}(\cite{F98}) \lb{mono-01} The monodromy map $\M_R^t$ is a M\"obius map, which is determined by the ordinary differential equation
    \be \lb{monodromy}
    M_R'(t)  =  -\frac{1}{2R}\begin{pmatrix}
0 & \xi'(t) + \eta'(t)i\\
\xi'(t) - \eta'(t)i & 0
\end{pmatrix}M_R(t)=:-\frac{1}{R} A(t)M_R(t),   \quad M(0) = \text{I},
    \ee
on    
    \[
  	SU(1,1) =  \left\{\begin{pmatrix}
	a & b\\
	\bar{b} & \bar{a}
	\end{pmatrix} 
	| a,b \in \mathbb{C}, {|a|}^{2} - {|b|}^{2} = 1\right\}.
    \]
More precisely, if
    \[
    M_{R}(t) = \begin{pmatrix} a_{R}(t) & b_{R}(t) \\ \overline{b_{R}(t)} & \overline{a_{R}(t)} \end{pmatrix}
    \]
is the solution of \x{monodromy}, then
    \be\lb{Mo}
    \M_R^t(z) =M_R(t)\cdot z:= \frac{a_R(t)z + b_R(t)}{\overline{b_R(t)}z + \overline{a_R(t)}}.
    \ee
for any $z \in \text{S}^1$.
   \end{thm}
   
\begin{proof}
     For fixed $R$, let $z(t)=e^{i\th(t)}$ where $\th(t)$ is a solution of the BCD Eq. \x{RSE}. Then
     \be\lb{Rec}
     \dot{z}=i\dot{\th}z=\frac{1}{2R} [(\xi'-i\eta')z^2-(\xi'+i\eta')],
     \ee
where the identities $\sin \th= 1/2i(z-z^{-1})$ and $\cos\th=1/2(z+z^{-1})$ are used. Notice that Eq. \x{Rec} is a Riccati equation. Introduce the ODE system
\be\lb{fun} X'(t)=-\frac{1}{R}A(t)X(t),\quad  X(t)=(u(t),v(t))^T,\ee
where $A(t)$ is defined in \x{monodromy}.

It is direct to check that $z(t)=u(t)/v(t)$ satisfies Eq. \x{Rec}. Hence the fundamental solution matrix $M_R(t)$ of \x{fun} recover the monodromy map $\M_R^t$ by the standard matrix correspondence to a M\"obius map \x{Mo}. Finally, since $\M_R^t$ maps $S^1$ to $S^1$, it follows that $M_R(t)\in SU(1,1)$.
\end{proof}\qed

Denote
\[\M_R:=\M_R^{2\pi},\quad  M_R:=M_R(2\pi)=\begin{pmatrix} a_{R} & b_{R} \\ \overline{b_{R}} & \overline{a_{R}} \end{pmatrix}.  \]
A non identity monodromy map $\M_R$ is classified into the following three cases:
\begin{enumerate}
    \item $\M_R$ is hyperbolic, i.e., has two fixed points on $S^1$ $\Leftrightarrow$ $|\text{tr}(M_R)|>2$;
    \item $\M_R$ is parabolic, i.e. has one fixed point on $S^1$ $\Leftrightarrow$ $|\text{tr}(M_R)|=2$;
    \item $\M_R$ is elliptic, i.e. has no fixed point on $S^1$ $\Leftrightarrow$ $|\text{tr}(M_R)|< 2$;
\end{enumerate}
where $\text{tr}(M_R)=2\text{Re} (a_R)$ denotes the trace of the matrix $M_R$.
\bigskip

\begin{defn} The set of $R$ such that the monodromy map $M_R:=M_R(2\pi)\in SU(1,1)$ is elliptic is called the elliptic set of $\Gamma$.
\end{defn}

The following lemma tells the relation between the types of the monodromy map and the rotation number, which is direct by definition.

\begin{lem} The rotation number $\rho(R)$ is in $\Z$ if and only if $\mathcal{M}_R$ is hyperbolic or parabolic; $\rho(R)$ modulo $\Z$ is in $(0,1)$ if and only if $\mathcal{M}_R$ is elliptic.

\end{lem}

It is known that, if a M\"obius map is elliptic, it is conjugated to a rotation. We summarize this property and use it to give the rotation number in the following proposition.
 \begin{prop} \lb{mono-02}
      If $\M_R$ is elliptic, then it is topologically conjugated to a rotation on $S^1$ and in any connected component $I_E$ of the elliptic set, the rotation number of BCD is given by
      \be\lb{ellip}\rho(R) = \frac{\epsilon}{\pi}\,\operatorname{arccos}(\frac 12\text{tr}(M_R))+k,\quad R\in I_E.\ee
      for some $k\in \Z$, $\epsilon\in \{1,-1\}$.
 \end{prop}

\Proof If $\M_R$ is elliptic, then $|\text{Re}(a_R)|<1$.
Let $z_0$ be the unique fixed point of $\mathcal{M}_R$ inside the unit open disk $\mathbb{D}$, i.e. $M_R\cdot z_0=z_0$. Define
$$
G=G_{z_0}\ :=\ \frac{1}{\sqrt{1-|z_0|^2}}\begin{pmatrix}1&-z_0\\ -\overline{z_0}&1\end{pmatrix}\in SU(1,1),\quad G_{z_0}\cdot z_0=0.
$$
So $GM_RG^{-1}\cdot 0=0$, it has to be of the diagonal form
\[
\hat{M}_R= \begin{pmatrix}
e^{i\th_R} & 0\\
0 & e^{-i\th_R}
\end{pmatrix}.
\]
That is, $\M_R$ is conjugated in $SU(1,1)$ to the M\"obius map
\[ \hat{M}_R\cdot z=e^{i 2\th_R}z.\]
It follows that the rotation number $\rho(R)$ modulo $\Z$ equals $2\th_R/2\pi$ .

On the other hand, since $e^{\pm i\theta_R}$ are the eigenvalues of the matrix $M_R$, they can be computed via
\[\lambda= \text{Re} (a_R)\pm  i\sqrt{1- (\text{Re}(a_R))^2},\]
hence $\th_R$ takes value in $\pm \operatorname{arccos} \!\left(\text{Re}(a_R)\right)+2k\pi $, where $\operatorname{arccos} \!\left(\text{Re}(a_R)\right)\in (0,\pi)$.

\qed

 \bb{thm}\lb{anal}
If $\rho(R) \notin \Z$, then the rotation number function $\rho$ is real analytic at $R$. In particular, $\rho(R)$ can not be locally constant in any connected component of the elliptic set.
    \end{thm}
\begin{proof}
     Suppose $\rho(R)\notin \Z$, then the monodromy map $\M_R$ has no fixed points, i.e. $\M_R$ is elliptic. By Proposition \ref{mono-02}, the rotation number is given by \x{ellip}. To prove that $\rho$ is analytic at $R$, it remains to show that $a_R$ is analytic in $R$.

     Denote $\lambda=1/R$, the monodromy equation \x{monodromy} writes as $ M'(t;\lambda)=-\lambda A(t)M(t;\lambda)$. By Picard iteration:
     \[M(t,\lambda)=\sum_{k=0}^{\infty}(-\lambda)^k M_k(t),\quad M_k(t)=\int_{0<t_1<\cdots<t_k<t }A(t_k)\cdots A(t_1)dt_1\cdots dt_k.\]
     Let $C:=\sup_{t\in[ 0,2\pi]} \|A(t)\|$, then
     \[\|M(t,\lambda)\|\leq \sum_{k=0}^{\infty} |\lambda|^kC^kt^k/k!=\exp(|\lambda| Ct).\]
     Thus $M(t,\lambda)$ is analytic in $\lambda$, whence in $R$ for $R\neq 0$.

     Therefore, the rotation number function $\rho(R)$ can not be locally constant in any connected component $I_E$ of the elliptic set. Otherwise, $\rho(R)$ will be constant in the whole $I_E$. But by the continuity of $\rho$, this implies $\rho(R)\in \Z$ for $R\in I_E$, which is a contradiction.
\end{proof} \qed
\bigskip

We end this section by an asymptotic expansion of $M_R$ using Magnus series, that is, the expansion for $\Omega_R:=\log M_R$. See the Appendix A for a brief introduction.

\begin{lem} \label{magnus-area} 
    For any piecewise $C^2$ closed curve $\Ga$ and $R > 0$, if $\ell(\Ga)/R$ is sufficiently small, then it admits the expansion for $\Omega_R=\log M_R$ :
    \[\Omega_R=\frac{\mathcal{A}(\Ga)}{R^2} J_0+ O(\ell(\Ga)^3R^{-3})\]
where $J_0=\mathrm{diag}(i/2,-i/2)$.
\end{lem}
\begin{proof}Let us denote $A_R=-\frac{1}{2R} A(t)$ where $A(t)$ is defined in \x{monodromy}, and the period of $\Ga$ is $L$.  Since 
\[
\int_{0}^{L}||A_{R}(t)||dt = \frac{1}{2R}\int_{0}^{L}||\Ga'(t)||dt  = \frac{\ell(\Ga)}{2R},
\]
here we use the norm $||A|| = \max_{||x|| = 1}{||Ax||}$. Hence, if $\ell(\Ga)/R$ is sufficiently small, it follows from the two-term Magnus expansion in Theorem \ref{thm:magnus-2} that
\[\Omega_{R}=\Omega_{R,1}+\Omega_{R,2}+O(\ell(\Ga)^3R^{-3}).\]
where
    \[\Omega_{R,1}(L)=\int_0^L A_R(s)ds=0.\]
    by periodicity of $\Gamma$. And 
    \[\Omega_{R,2}(L)=\frac{1}{2}\int_{0<s_2<s_1<L}[A_R(s_1),A_R(s_2)]ds_2ds_1.\]
    A direct computation gives
    \[[A(s_1),A(s_2)]=4(\xi'(s_1)\eta'(s_2)-\eta'(s_1)\xi'(s_2))J_0.\]
The ordered double integral reduces to a line integral:
\[
\int_{0<t_2<t_1<L}(\xi'_1\eta'_2-\eta'_1\xi'_2)\,dt_2dt_1
=\int_0^{L}\big(\xi\,\eta'-\eta\,\xi'\big)\,dt
=\oint_{\Ga}\xi\,d\eta-\eta\,d\xi
=2\,\mathcal{A}(\Ga).
\]
Hence
\[
\Omega_{R,2}(L)=\frac {1}{R^2}\mathcal{A}(\Ga)J_0.
\]

\end{proof}


\ifl
Geometrically, system \x{SE} means that $\vr'$ is the projection of $\Ga'$ in the direction $\vr-\Ga$. The shadower $\vr(t)$ will point to (resp. oppose to) the escaper $\Ga(t)$ when $\alpha(t)>0$ (resp. $\alpha(t)<0$). When $\alpha(t)=0$ or $\Ga'(t) \cdot(\vr(t) -\Ga(t)) =0$, the shadower $\vr(t)$ will stop at these times.

With these explanations to the strategy that the shadower is always staring at the escaper, system \x{SE} of ODEs is called in this paper the {\it shadowing equation} (SE, for short) to (the EC) $\Ga$, and, meanwhile, the solutions of SE \x{SE} are called the {\it shadowing curves} (SC or SCs, for short) to $\Ga$ or to $\Ga(t)$. Moreover, for any SC $\vr(t)$ to $\Ga(t)$,
    \be \lb{sdr1}
    R:= \|\Ga(0)-\vr(0)\|\equiv \|\Ga(t)-\vr(t)\|>0
    \ee
is called the {\it bicycle length} of SC $\vr(t)$.

    \bb{thm}\lb{sys}
Let the EC $\Ga$ be given. Then, for any initial point $\vr(0)$ different from $\Ga(0)$, SE \x{SE} admits a unique globally defined solution
    \be \lb{vrtt}
    \vr=\vr(t)
    =\vr(t,\vr(0)), \qq t\in \R.
    \ee
Hence any solution \x{vrtt} defines a parameterized SC to $\Ga(t)$ with the bicycle length $R$ being defined by \x{sdr}.
    \end{thm}
\fi

\ifl
\hrule

In this paper, we mainly concentrate on the studying for shadowing curves on the Euclidean plane $\R^2$ when the escaping curves are planar closed curves. The content and results are as follows.

In \S \ref{second}, we will first briefly study the invariance properties on shadowing problems. Then we will deduce an extending shadowing equation \x{ESE} for general dimension which is a higher dimensional linear system. Finally, when the planar shadowing curves to planar escaping curves are considered, we will use the moving polar coordinates to deduce a reduced shadowing equation \x{RSE} which is a nonlinear differential equation on the circle.

The main content is given in \S \ref{third}. When the escaping curve is a planar closed curve $\Ga$ with some regularity, we use the reduced shadowing equation to introduce the rotation number $\ro(R)=\ro_\Ga(R)$ from dynamical systems theory \cite{A83, H80, KH95}, which is a function of the bicycle length $R\in(0,+\oo)$. It is proved in Lemma \ref{inv} that $\ro_\Ga(R)$ is independent of the parameterizations of $\Ga$, i.e. $\ro_\Ga(R)$ depends only on the geometry of $\Ga$. Moreover, we find that $\ro_\Ga(R)$ has closed connections with the perimeter of $\Ga$, the rotation index of  $\Ga$, and the area enclosed by $\Ga$. For details, see Lemma \ref{upbound}, Theorem \ref{rho-01} and Theorem \ref{mono}. By using rotation number $\ro(R)$, we will apply the dynamical behavior of circle diffeomorphisms,  including the Denjoy theorem, to give a fair complete characterization of types of planar shadowing curves. The main results are stated in Theorem \ref{main1} and Theorem \ref{main2}. Typically, we have the following three types of shadowing curves:
    \begin{itemize}
\item when the bicycle length $R$ is not too large, the shadowing problem admits only $2\pi$-periodic shadowing curves and those shadowing curves which are approaching to periodic ones.
\item when $R$ is large enough and $\ro(R)$ is rational, the shadowing problem admits subharmonic ($2p\pi$-periodic) shadowing curves and those shadowing curves which are approaching to subharmonic ones.
\item  when $R$ is large enough and $\ro(R)$ is irrational, \ysb{each shadowing curve is dense in the possible shadowing domain} $\D_R$ as in \x{Dr}, which is like an `annulus' depending only on $R$.
    \end{itemize}

In order to distinguish the ranges of these different bicycle lengths, we use the properties of rotation number $\ro_\Ga(R)$ to introduce two notions which are called the critical bicycle length $\mcr(\Ga)$ and the turning bicycle length $\ol{R}(\Ga)$. See Definition \ref{CSD} and Definition \ref{TSD} respectively. These notions depend only on the geometry of closed escaping curves $\Ga$.
It is proved in Theorem \ref{circ} that $\mcr(\Ga)$ is really different from the `perimeter' when $\Ga$ is not a circle.

In \S \ref{fourth}, by considering the unit circle as an escaping curve, we will examine all shadowing curves in details using the reduced shadowing equation. For this simplest example, the shadowing problem will admit beautiful shadowing curves. These will be plotted in Figures \ref{sccircles}-\ref{Erg4}. When the escaping curve is chosen an ellipse, we give some analytic and numerical analysis to possible shadowing curves.

In \S \ref{fifth}, we impose a conjecture, which asserts that the critical bicycle length $\mcr(\Ga)$ and the turning bicycle length $\ol{R}(\Ga)$ are coincident: $\mcr(\Ga)=\ol{R}(\Ga)$ for typical closed curves $\Ga$. Once this is true, we can give a compete characterization to all types of shadowing curves. As mentioned in \S \ref{third}, these quantities are related with the geometric properties of $\Ga$. Hence the conjecture may be of independent interest from the point of view of differentiable geometry.

Finally, although the paper contains several interesting results, it is just a beginning study for shadowing problems. Moreover, most of the proofs in this paper are not difficult from the point of view of dynamical systems.

\fi

\ifl Hence all RTs to circles can be obtained explicitly. One can find from \cite{M20, WZ22} that this simple example admits very interesting RTs.

For our purpose, let us calculate the rotation number $\ro(R)$ of Eq. \x{rse2}. Because of the relations \x{rse2}---\x{phi}, we first calculate the rotation number $\ro_\phi(R)$ of Eq. \x{phi}.

Case 1: $R\in(0,1)$. In this case, Eq. \x{phi} has two geometrically different equilibria
    \[
    \phi_R^+ = -\phi_R^-:=-\arccos(-R)\equiv -\z(\pi/2+\arcsin R\y).
    \footnote{These equilibria yield $2\pi$-periodic RTs $\ga_{R}^\pm(t) =\sqrt{1-R^2} \z( \cos (t\mp \arcsin R), \sin (t\mp \arcsin R)\y)$.}
    \]
These equilibria $\phi_R^\pm$ are constant solutions $\phi^\pm(t)\equiv \phi^\pm_R$ of Eq. \x{phi}. Hence, by the definition of rotation numbers, one has
    \be \lb{rn1}
    \ro_\phi(R) = \lim_{t\to \oo}\f{\phi^\pm(t)- \phi^\pm_R}{t}=0
    \ee

Case 2: $R=1$. In this case, the two equilibria shrink into a single equilibrium $\phi=\pi$. One still has $\ro_\phi(1)=0$.

Case 3: $R\in(1,+\oo)$. In this case, Eq. \x{phi} has no equilibrium because $\cos\phi+R>0$ for all $\phi\in\R$. Any solution $\phi(t)$ of Eq. \x{phi} satisfies
    \bea \lb{R3}
    -t \EQ -\int_0^t \dt = \int_{0}^{t} \f{R\phi'(t)}{\cos \phi(t)+R} \dt \nn\\
    \EQ \int_{\phi(0)}^{\phi(t)} \f{R}{\cos \vp+R} \rd\vp \nn\\
    \EQ H(\phi(t)) - H(\phi(0)),\qq t\in \R,
    \eea
where
    \[
H(\phi):= \int_0^\phi \f{R\rd\vp}{\cos \vp+R}, \qq \phi \in \R,
    \]
is a strictly increasing, odd diffeomorphism of $\R$. Explicitly, by introducing a constant
    \be \lb{rho}
    \rho = \rho_R:= \f{\sqrt{R^2-1}}{R}\in (0,1),
    \ee
one has
    \[\bb{split}
    H(\phi) & \equiv \f{2}{\rho} \arctan\z( \sqrt{\f{R-1}{R+1}} \tan \f{\phi}{2}\y)\qqf \phi\in(-\pi,\pi),
    \\
    H(k \pi) & = {k\pi}/{\rho}\qqf k\in \Z,
    \end{split}
    \]
and
    \be \lb{sec34}
    H(\phi+2k\pi) \equiv H(\phi) +  {2k\pi}/{\rho} \qqf k\in \Z, \ \phi\in\R.
    \ee
Now we have from \x{R3}
    \[
    \phi(t) \equiv H^{-1}\z( H(\phi(0))-t\y)\qqf  t\in \R.
    \]
The property \x{sec34} for $F$ is transferred to
    \[
    \phi(t+ 2k\pi/\rho) \equiv \phi(t)-2k\pi\qqf k\in \Z, \ t\in\R.
    \]
In particular, one has
    \[
    \f{\phi(2k\pi/\rho) -\phi(0)}{2k\pi/\rho}\equiv -\rho.
    \]
Therefore
    \be \lb{rn3}
    \ro_\phi(R) = -\rho, \qq R\in(1,+\oo).
    \ee

For general case $R\in(0,+\oo)$, one has from  \x{ths} that
    \(
    \th(t) = t+ \phi(t).
    \)
Thus
    \[
    \ro(R) = \lim_{t \to \oo} \f{\th(t)-\th(0)}{t} = 1 + \lim_{t \to\oo} \f{\phi(t)-\phi(0)}{t}= 1+\ro_\phi(R).
    \]
By using results \x{rn1}, \x{rho} and \x{rn3}, we have the following conclusion.
\fi

\section{Two new geometric invariants} \lb{prn}

Given a planar $C^2$ closed curve $\Gamma$, there are classical geometric invariants that are related by the rotation number function $\rho_{\Gamma}(R)$. We recall them in the following:
\begin{itemize}
    \item the perimeter:
    \[\ell=\ell(\Ga):=\int_0^{2\pi }\|\Ga'(t)\|\dt.\]
    \item the rotation index :
    \[\omega=\omega(\Gamma)=\f{1}{2\pi}\oint_{\Ga} \f{x\dy- y\dx}{x^2+y^2}\in \Z.\]
    \item the {\it algebraic area} enclosed by $\Ga$:
     \[\mathcal{A}= \A(\Ga):= \oint_\Ga x\dy := - \oint_\Ga y\dx =  \f{1}{2}\oint_\Ga x\dy-y\dx.\]
     \item the minimum and maximum radius of $\Ga$:
     \[R_{\min}(\Ga)=1/\max_t |\kappa(t)|,\quad R_{\max}(\Ga)=1/\min_t |\kappa(t)|.\]
where $\kappa(t)$ is the curvature function of $\Ga$.
\end{itemize}

For example, for closed curves $\mcc_n(t)=(\cos (n t), \sin (n t))$, $n\in \Z$, one has the  perimeters $\ell(\mcc_n) = 2|n| \pi$ and the algebraic areas $\A(\mcc_n) = n \pi$. For a directional figure-eight curve $F_8$ formed by two circles touching at the origin, the algebraic area $\A(F_8)$ can be any prescribed number by varying the areas of the two disks.
\bigskip

Since the rotation number $\rho_{\Gamma}(R)$ is defined independent of the parametrization of $\Gamma$, we will introduce two new geometric invariants that encode the bicycle dynamics from $\rho_{\Gamma}$:
\begin{itemize}
    \item the \emph{critical B-length} $\mcr(\Gamma)$:  the right end of the first plateau  of $\rho_\Gamma(R)$.
    \item the \emph{turning B-length} $\overline{R}(\Gamma)$ :  the left endpoint of the maximal interval on which $\rho_\Gamma(R)$ is strictly monotone.
\end{itemize}

\begin{prop}\lb{smallR}
Let $\Ga$ be a regular $C^2$ closed curve and $R_{\min}$ denote its minimal radius of curvature. Then
    \be \lb{rn-s01}
    0<R \le R_{\min} \q \Longrightarrow \q \ro_{\Ga}(R) \equiv \om=\om(\Ga).
    \ee
\end{prop}
\begin{proof}
    We will consider the BCD equation \x{eq:angle} for the steering angle $\alpha$:
    \[\alpha'= B(t)\left(\frac{1}{R}\sin\alpha-\kappa(t)\right)=:f_R(t,\alpha)\]
 For $R\in (0,R_{\min})$, the vector field
    \[f_R(t,-\pi/2)=f_R(t,3\pi/2)<0,\quad f_R(t,\pi/2)>0\]
    It follows that the Poincar\'e map $\tilde{\mathcal{P}}: \alpha_0\mapsto \alpha(2\pi;\alpha_0)$ satisfies
    \[\tilde{\mathcal{P}}_R(\pi/2, 3\pi/2)\subset (\pi/2,3\pi/2),\quad \tilde{\mathcal{P}}_R^{-1}(-\pi/2,\pi/2)\subset (-\pi/2,\pi/2)\]
hence $\tilde{\mathcal{P}}_R$ has two fixed points, and $\tilde{\rho}(R)=0$. Therefore, $\rho(R)=\tilde{\rho}(R)+\omega=\omega$ for $R\in (0,R_{\min})$.
\end{proof}

    \bb{defn} \lb{CSD}

Suppose that $\Ga$ is a regular $C^2$ closed curve such that $\om=\om(\Ga) \ne 0$. Then  the following quantity
    \be \lb{csd}
    \underline{R}=\underline{R}(\Ga):= \sup \z\{R_*\in(0,+\oo): \ro_\Ga(R) \equiv \om \mbox{ for all } R\in (0,R_*] \y\}\in  (0,+\oo)
    \ee
is well-defined. It is called the {\it critical B-length} of $\Ga$.
    \end{defn}

 \begin{prop} \lb{circ}
Let $\Ga$ be a regular $C^2$ closed curve with $\omega=\omega(\Ga)\ne 0$. Then an upper bound for the critical B-length is given by
    \be \lb{ulr-u1}
    R_{min}(\Gamma)\leq \ul{R}(\Ga)\leq \frac{\ell(\Ga)}{2\pi |\omega|},
    \ee
where the right hand-side equality holds if and only if $\Ga $ is a circle.
    \end{prop}

\Proof The left inequality follows from Proposition \ref{smallR}.
By reversing time if necessary, we assume that $\om>0$. For $R\leq \ul{R}(\Gamma)$, one has $\ro(R)=\omega$. It follows from Lemma \ref{upbound} that
    \[
    \omega=\ro(R)\leq \frac{\ell}{2\pi R},\quad R\leq \mcr(\Gamma)
    \]
Hence $\mcr(\Gamma)\leq \ell/2\pi\om$.

Now we consider the case $R^*=\ell/2\pi\om$. Take the normalized arc-length parametrization.
for $\Ga$ and plug it into the BCD Eq. \x{eq:angle}. We obtain
    \be\label{rse411}\alpha'(s)=\frac{2\pi \omega}{\ell}\sin\alpha-\kappa(s),\ee
Integration from $0$ to $2\pi$ in \x{rse411} gives
    \[
    \alpha(2\pi,\alpha_0)-\alpha_0=\frac{2\pi \omega}{\ell}\int_0^{2\pi}(\sin \alpha(s,\alpha_0)-1)ds \leq 0.
    \]
It follows that the rotation number $\tro(R^*)\leq 0$. In particular, the Poincar\'e map $\tilde{\mathcal{P}}_{R^*}$ has a fixed point if and only if $\alpha(s,\alpha_0)\equiv \pi/2$ whence $\kappa(s)\equiv 2\pi\omega/\ell$.
, i.e. $\Gamma$ is a circle.\qed

\begin{thm} [Isoperimetric-type inequality] If $\Ga$ is a $C^2$ strictly convex closed curve, then
 \be\lb{isoine} \sqrt{\frac{\A(\Ga)}{\pi}}\leq \mcr(\Ga)\leq \frac{\ell(\Ga)}{2\pi}.\ee
 where each equality in \x{isoine} holds if and only if $\Ga$ is a circle of radius $\mcr(\Ga)$.
\end{thm}
\begin{proof}
    Suppose $\Ga$ is a strictly convex simple closed curve, then $|\om|=1$ and $\kappa>0$.


The left inequality in \x{isoine} comes from Menzin's conjecture for strictly convex curves: if $\A(\Ga)>\pi R^2$, then the monodromy map $\M_R$ is hyperbolic (\cite{FLS13}). Indeed, $\mcr(\Gamma)$ is the smallest length such that the monodromy turns from hyperbolic to parabolic, thus
\[\mcr(\Ga)\geq \sqrt{\frac{\A(\Gamma)}{\pi}}.\]
The equality holds, if and only if the support function of the closed rear track at $\mcr(\Gamma)$ is of the form $p(\varphi)= a\cos\varphi+b\sin\varphi$ (see \cite{FLS13}), which implies $\gamma$ is a single point, and $\Ga$ is a circle of radius $\mcr$ centered at $\ga$.

\end{proof}\qed

 \bb{prop} \lb{convex}
   Assume that $\Gamma$ is a $C^2$ strictly convex closed curve with $\omega(\Gamma)=1$. Then
   \begin{enumerate}
    \item The monodromy map $\mathcal{M}_R$ is hyperbolic for $R \in (0,\mcr(\Gamma))$.
    \item  For all $R > \mcr(\Gamma)$, one has $\rho(R) < 1$.
   \end{enumerate}

   \end{prop}

\Proof Write $\mcr=\mcr(\Gamma)$. We will work on the BCD equation \x{eq:angle}.  Notice that, any periodic solution curve of \x{eq:angle} must be restricted in the strip $S:=[0,2\pi]\times (0,\pi)$ in $(t,\alpha)$-plane. Indeed,  for $\alpha\in [\pi,2\pi]$, the vector field $f_R(t,\alpha)<0$, solutions are strictly decreasing, hence a periodic solution  $\alpha(t;\alpha_0)$ has to start with initial position $\alpha_0\in (0,\pi)$; and the solution curve is trapped in $S$  since at the boundary, $f_R(t,0)=f_R(t,\pi)<0$. Also,
    \begin{equation}\label{eq:cmp}
        f_R(t,\alpha)<f_{R'}(t,\alpha),\quad \alpha\in (0,\pi),\quad R>R'.
    \end{equation}

  1. For $R\in (0,\mcr)$, by \x{eq:cmp} the vector fields $f_R$ at the periodic solution curve $\{t,\alpha_{\mcr}(t,\alpha_0)\}$ satisfies
\[f_R(t,\alpha_{\mcr}(t,\alpha_0))>f_{\mcr}(t,\alpha_{\mcr}(t,\alpha_0)).\]
Combining with the fact that $f_R(t,0)=f_R(t,\pi)<0$, one gets
\[\tilde{\mathcal{P}}_R(\alpha_0,\pi)\subset (\alpha_0,\pi),\quad \tilde{\mathcal{P}}_R^{-1} (0,\alpha_0)\subset (0,\alpha_0).\]
Thus $\tilde{\mathcal{P}}_R$ has two fixed points.

2. By definition, $\tilde{\rho}(\mcr)=0$. Let $\alpha_{\mcr}(t;\alpha_0)\in (0,\pi)$ be a $2\pi$-periodic solution.
    Suppose $R>\mcr$, by  (\ref{eq:cmp}), the solution curve $\alpha_R(t;\alpha_0)$ lies below
    $\alpha_{\mcr}(t;\alpha_0)$ for $t$ small  and can not cross it for any $t\in [0,2\pi]$. Hence the Poincar\'e map of \x{eq:angle} verifies
    \begin{equation}\label{eq:small}
        \tilde{\mathcal{P}}_R(\alpha_0)\leq \tilde{\mathcal{P}}_{\mcr}(\alpha_0)=\alpha_0.
    \end{equation}

    It follows that $\tro(R)\leq 0$ for $R>\mcr$. Suppose that there is $R_1>\mcr$ such that $\tro(R_1)=0$, that implies, there exists a $2\pi$-periodic solution $\alpha_{R_1}(t;\alpha_1)$ which lies in $(0,\pi)$. By (\ref{eq:cmp}), for all $R<R_1$, the solution curve $\alpha_{R}(t;\alpha_1)$ always lies above $\alpha_{R_1}(t;\alpha_1)$ for $t\in [0,2\pi]$. In particular, $\tilde{\mathcal{P}}_R(\alpha_1)\geq \alpha_1$ and $\tro(R)\geq 0$. Thus $\tro(R)=0$ for $R\in (\mcr,R_1)$, which contradicts with the definition of $\mcr$. Hence $\tro(R)<0$ for all $R>\mcr$.
\qed

  \bb{prop} \lb{asym}
Assume that $\Ga$ is a regular $C^2$ closed curve. Then
    \be \lb{rho-ra}
    \ro(R) = \f{\A(\Ga)}{2\pi R^2} +O\z(\f{1}{R^3}\y)\qq \mbox{as } R\to +\oo,
    \ee
where $\A(\Ga)$ is the algebraic area enclosed by $\Ga$.  In particular, $\ro(R)$ is strictly decreasing (resp. increasing) if $\A(\Ga)>0$ (resp. $\A(\Ga)<0$) for $R$ large enough.%
    \end{prop}

\Proof We use the Magnus expansion to study the asymptotic behavior of $\ro(R)$ as $R\to +\oo$. By Lemma \ref{magnus-area}, the monodromy map $M_R$ is given by
\[M_R=\exp (h(R)J_0+O(R^{-3})),\qquad h(R)= \frac{\mathcal{A}(\Ga)}{R^2}.  \]
where $J_0=\mathrm{diag}(i/2,-i/2)$, and $O(R^{-3})$ is a $2\times 2$ matrix whose matrix norm being $O(R^{-3})$.

The eigenvalues of $M_R$ are given by
\[\lambda_\pm=\exp\left(\pm i\frac{\mathcal{A   }(\Ga)}{2 R^2}+O(R^{-3})\right).\]
Thus the rotation number $\ro(R)$ satisfies
\[\ro(R)=\frac{2\arg(\lambda_+)}{2\pi}=\frac{\mathcal{A}(\Ga)}{2\pi R^2}+O(R^{-3}).\]
This proves \x{rho-ra}. In particular, for $R$ large enough, $\rho(R) \to 0$, hence $\rho(R)$ is analytic by Corollary \ref{anal} and
\[\rho'(R)=-\frac{\mathcal{A}(\Ga)}{\pi R^3}+ O(\frac{1}{R^4})\]
It follows the strict monotonicity of $\rho$ for $R$ large enough, depending on the sign of $\A$ .

\qed

\bigskip

Let us introduce the following hypothesis $\bf{(H)}$ on regular $C^2$ closed curve $\Ga$:
    \[
    {\bf{(H)}}: \qq \om=\om(\Ga) \ne 0,
    \andq \A=\A(\Ga)\ne 0.
    \]

    \bb{rmk} \lb{Hy}
The hypothesis $\bf{(H)}$ is always verified if $\Ga$ is a non trivial Jordan curve (a simple closed curve).  On the contrary,  it can be the case  that $\omega_0=0$ if $\Ga$ is a figure eight formed by two circles touching at the origin, and $\mathcal{A}_0$ any prescribed number by varying the areas of the two disks.
    \end{rmk}

    \bb{defn}\lb{TSD}
{\rm
For any  regular $C^2$ closed curve $\Ga$ satisfying hypothesis $\bf{(H)}$, we define the {\it turning B-length}
to $\Ga$ as
    \[
    \overline{R}=\overline{R}(\Ga):=\inf \{R^*\in [\underline{R},+\infty):  \ro(R)\neq \omega \,\text{and is monotone in }\, R\in (R^*,+\infty)\}.
    \]
}
    \end{defn}

One sees that
    \be \lb{RR0}
    0<\ul{R}(\Ga) \le \overline{R}(\Ga)<+\oo
    \ee
for any $C^2$ regular closed curve $\Ga$ satisfying hypothesis $\bf{(H)}$.

  \bb{exa} \lb{exa1}

Suppose that the front tire track is an ellipse
    \[\E_b=(\cos t, b \sin t),\quad b>0\]
The curvature function of $\E_b$ is
\[\kappa(t)=\f{b}{\z(\sqrt{\sin^2 t + b^2 \cos^2 t}\y)^3}.\]

 Let us choose $b=2$. One has
    \[
    R_{\min}(\E_2) =\f{1}{2}, \qq R_
    {\max}(\E_2) = 4,\quad  \sqrt{\f{\A(\E_2)}{\pi}}=\sqrt{2} , \qq \f{\ell(\E_2)}{2\pi} =\f{4 {\mathbb E}(3/4)}{\pi}
    \]

We have plotted in Figure \ref{rnb2} the function of rotation numbers $\ro_2(R)$.  Numerically, the critical B-length of $\E_2$ is
    \[
    \mcr(\E_2)=\overline{R}(\E_2)\thickapprox 1.445.
    \]
    \end{exa}

\begin{figure}[ht]
\centering
\includegraphics[width=8cm, height=5cm]{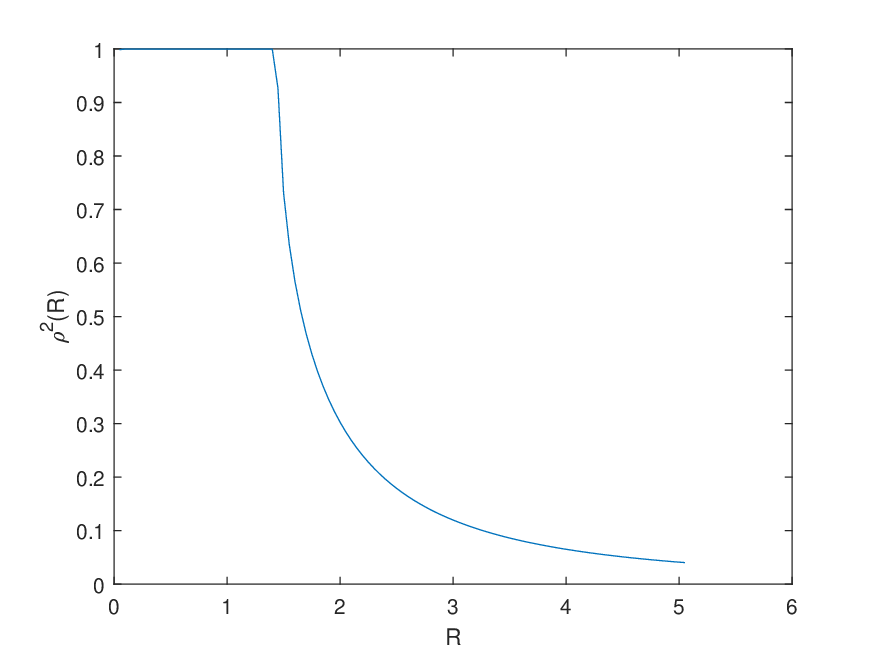}
\caption{Rotation numbers $\ro_b(R)$, $b=2$, as a function of $R$. }
\label{rnb2}
\end{figure}

\section{Uniqueness of plateau for star-shaped curves}
\setcounter{equation}{0} \lb{fifth}
In the previous sections, we have seen that the rotation number function $\rho_{\Gamma}$ has a first plateau in $(0,\mcr(\Gamma))$ and strictly decreases (if $\A(\Gamma)>0$) for $R>\overline{R}(\Gamma)$. A natural question is that for which curves there is a unique plateau, which is the case for circles whose rotation number function could be explicitly computed, and for ellipses with numerical simulation.

 We refer the following conjecture in \cite{WZ22}:
    \bb{conj} [\cite{WZ22}]\lb{conj1}
For any smooth strictly convex closed curve $\Ga$ on the plane, there holds
    \[
    \ul{R}(\Ga)= \ol{R}(\Ga).
    \]
    \end{conj}

  The coincidence of the two invariants $\underline{R}(\Gamma)$ and $\overline{R}(\Gamma)$ is closely related to the conjecture about the
  types of the monodromy maps raised in \cite{BHT24}.
    \bb{conj} [\cite{BHT24}] \lb{conj2}
For any smooth strictly convex closed curve $\Ga$, there exists a number $R_0$ such that if $R<R_0$ then the monodromy $M_R$ is hyperbolic, and if $R>R_0$ it is elliptic.
    \end{conj}
  
We will show that Conjecture \ref{conj1} implies \ref{conj2} and that they are true for a larger class of curves: the star-shaped curves. 

A curve is star-shaped with respect to a point $z_0$ if every ray originating from $z_0$ intersects the curve exactly once.

Define  the {\it areal velocity} of a curve $\Gamma$  by 
\[a(t)=a_{\Gamma}(t):=\det (\Gamma(t),\Gamma'(t)).\]

\begin{lem}
    Let $\Gamma:[0,T]\to\mathbb R^2\!\setminus\!\{0\}$ be a $C^1$ closed curve with rotation index $\omega(\Gamma)=1$. Then the following are equivalent:
 \begin{enumerate}
    \item $\Gamma$ is star-shaped around the origin;
    \item The area velocity $a(t)$ of $\Gamma$ is positive, i.e. $a(t) > 0$ for all $t$.
 \end{enumerate}

\end{lem}

\begin{proof}
    Suppose $\Gamma$ is star-shaped around the origin, then $\Gamma$ has a parametrization\[\Gamma(t)=r(t)(\cos t,\sin t),\quad t\in [0,2\pi]\]
    then
    \[a(t)=\det (\Gamma(t),\Gamma'(t))= r(t)^2 > 0.\]
    On the other hand, suppose $\Gamma(t)$ is a curve with positive areal velocity. Write $\Gamma(t)$ in polar coordinates:
    \[\Gamma(t)=r(t)(\cos\theta(t),\sin \theta(t)).\]
    then 
    \[a(t)= r(t)^2\theta'(t).\]
Positive areal velocity implies $\theta'(t)> 0$. Because $\omega(\Gamma)=+1$, the total angle increase is
$\theta(T)-\theta(0)=2\pi$. Hence $\theta$ is an increasing continuous lift whose range has length
$2\pi$, i.e. $\theta$ provides a polar–angle parametrization of $\Gamma$, which is the star-shaped property. 

\end{proof}
\qed

\bigskip
We recall some basic facts about the Lie group $SU(1,1)$ and its Lie algebra $\mathfrak{su}(1,1)$. Denote $J=\mathrm{diag}(1,-1)$, then
\[SU(1,1)=\{S\in GL_2(\C)|S^*JS=J\},\]
and the matrices in $SU(1,1)$ preserve the Hermitian product
\[\langle u,v\rangle_J= u^*Jv.\]
where $^*$ denotes the Hermitian conjugate.

The Lie algebra $\mathfrak{su}(1,1)$ is composed of matrices of the form
\begin{equation}\label{eq:xform}
    X = \begin{pmatrix} i\alpha & \beta + i\gamma \\ \beta - i\gamma & -i\alpha \end{pmatrix},\quad \alpha,\beta,\gamma\in\mathbb{R}.
\end{equation}

Let
\begin{equation}\label{eq:generators}
J_0=\frac{i}{2}\begin{pmatrix}1&0\\0&-1\end{pmatrix},\quad
S=\frac12\begin{pmatrix}0&1\\1&0\end{pmatrix},\quad
T=\frac12\begin{pmatrix}0&i\\-i&0\end{pmatrix}.
\end{equation}
Then $\mathfrak{su}(1,1)=\mathrm{span}_\R\{J_0,S,T\}$, and
\[[J_0,S]=T,\quad [J_0,T]=-S,\quad [S,T]=-J_0.\]
Equip $\mathfrak{su}(1,1)$ with the the bilinear form
\[
\langle X,Y\rangle:=-\frac12\,\tr(XY),
\]
for which 
\begin{equation}\label{eq:gener}
    \langle J_0,J_0\rangle=\tfrac14>0,\quad \langle S,S\rangle=\langle T,T\rangle=-\tfrac14,\quad 
\langle J_0,S\rangle=\langle J_0,T\rangle=\langle S,T\rangle=0.
\end{equation}
The bilinear form is $\Ad$-invariant with the adjoint action $\text{Ad}_g X= gXg^{-1}$.

\medskip




\bigskip

Denote $\lambda=1/R$. Consider the monodromy ODE on $SU(1,1)$:
\[U_\lambda'(t,s)=-\lambda\,A(t)\,U_\lambda(t,s),\ U_\lambda(s,s)=I,
\]
where 
\[
A(t)=\xi'(t)\,S+\eta'(t)\,T.\]

Denote the monodromy matrix by $M_{\lambda}:=U_{\lambda}(2\pi,0)$ . On any connected component of the elliptic set, by Proposition \ref{mono-02}, there exists $G_{\lambda}\in SU(1,1)$ and an analytic function $\theta(\lambda)$, such that 
\[G_{\lambda} M_{\lambda} G_{\lambda}^{-1}=\exp(\theta(\lambda) 2J_0).\]  Define
    \[\widehat{J}(\lambda)=G_{\lambda}^{-1} (2J_0) G_{\lambda}\in \mathfrak{su}(1,1) .\]
 then $M_{\lambda}$ writes as
 \[
M_\lambda=\exp\!\big(\theta(\lambda)\,\widehat J(\lambda)\big),\qquad
\langle \widehat J(\lambda),\widehat J(\lambda)\rangle=1.
\]

\begin{prop} \label{main1}Suppose $\Gamma$ is a $C^2$ star-shaped closed curve with $\omega(\Gamma)=1$. Then on any connected component $I_E$ of the elliptic set,
\[
\theta'(\lambda)=\lambda \int_{0}^{2\pi}\!a(s)\big\langle
\Ad_{\,U_\lambda(2\pi,s)^{-1}}\widehat J(\lambda),\,J_0\big\rangle\,ds,\quad \lambda\in I_E.
\]
In particular, $\theta(\lambda)$ is strictly increasing on $I_E$:
\[\theta'(\lambda)\geq \lambda \mathcal{A}(\Gamma),\quad \forall \lambda\in I_E.\]

\end{prop}

\begin{proof}
   
\emph{Step 1.} Recall that the derivative of a matrix exponential $e^{A(\lambda)}$ is given by:$$\frac{d}{d\lambda} e^{A(\lambda)} = \int_0^1 e^{sA(\lambda)} A'(\lambda) e^{(1-s)A(\lambda)} ds$$
Plug in $A(\lambda)=\theta(\lambda)\widehat{J}(\lambda)$,
$$\partial_{\lambda}M_{\lambda}M_{\lambda}^{-1} = \int_0^1 \exp(s\theta\hat{J}) \left( \theta'\hat{J} + \theta\hat{J}' \right) \exp(-s\theta\hat{J}) ds$$
Take the
$\langle\cdot,\cdot\rangle$ pairing with \(\widehat J\):
\beaa 
\langle\widehat{J},\partial_{\lambda}M_{\lambda}M_{\lambda}^{-1}\rangle &=& \theta' \langle \widehat{J}, \widehat{J} \rangle + \theta \int_0^1 \langle  \Ad_{\exp(-s\theta\widehat{J}) }\widehat{J},\widehat{J}' \rangle ds\\
&=&\theta' \langle \hat{J}, \hat{J} \rangle + \theta \langle\widehat{J},\widehat{J}'\rangle.
\eeaa
Since \(\|\widehat J\|=1\) is constant,
\(\langle\widehat J,\widehat J'\rangle=\tfrac12\partial_\lambda\langle \widehat J, \widehat J\rangle=0\), we get
\begin{equation}\label{eq:tag1}
    \theta'(\lambda)=\big\langle \widehat J(\lambda),\,\partial_\lambda M_\lambda\,M_\lambda^{-1}\big\rangle.
\end{equation}

\emph{Step 2.}
Denote\(B(t,\lambda)=-\lambda A(t)\). Differentiate $U(t,s)U(s,t)=I$ in $t$, one has
\[\partial_t U_{\lambda}^{-1}(t,s)=\partial_t U_{\lambda}(s,t)=-U_{\lambda}(s,t)B(t),\]
the left trivialized derivative \(X_{\lambda}(t,s):=\partial_\lambda U_\lambda(t,s)\,U_\lambda(t,s)^{-1}\)
solves \[\partial_t X_{\lambda}=\partial_\lambda B+[B,X], \quad X(s,s)=0.\]
Conjugating $X_{\lambda}$ by \(U_\lambda\), taking derivative  with respect to  $t$ and then integrating from $s$ to $t$, yields
\[
\partial_\lambda U_\lambda(t,s)\,U_\lambda(t,s)^{-1}
=-\int_s^{t}\!\Ad_{\,U_\lambda(t,\tau)}A(\tau)\,d\tau.
\]
For \(t=2\pi,s=0\),
\begin{equation}\label{eq:tag2}
    \partial_\lambda M_\lambda\,M_\lambda^{-1}
=-\int_{0}^{2\pi}\!\Ad_{\,U_\lambda(2\pi,s)}A(s)\,ds.
\end{equation}
Plug \x{eq:tag2} into \x{eq:tag1} and use Ad-invariance of \(\langle\cdot,\cdot\rangle\):
\begin{equation}\label{eq:tag3}
    \theta'(\lambda)=-\int_0^{2\pi}\!\Big\langle
\Ad_{\,U_\lambda(2\pi,s)^{-1}}\widehat J(\lambda),\,A(s)\Big\rangle ds.
\end{equation}

\emph{Step 3.}
Define \(C(s):=\xi(s)\,S+\eta(s)\,T\), so that \(C'(s)=A(s)\) and

\[[A(s),C(s)]=(\xi(s)\eta'(s)-\xi'(s)\eta(s))\,J_0=a(s)J_0.\] 
Set \[Y(s):=\Ad_{\,U_\lambda(2\pi,s)^{-1}}\widehat J(\lambda).\]
Differentiate \(Y(s)\):
\[
Y'(s)=-\lambda\,[A(s),Y(s)]\qquad(\text{from }\partial_s U_\lambda(2\pi,s)=\lambda U_\lambda(2\pi,s)A(s)).
\]
Integrate \x{eq:tag3} by parts on \([0,2\pi]\):
\[
\theta'(\lambda)
=-\big[\langle Y,C\rangle\big]_{0}^{2\pi}
+\int_0^{2\pi}\!\langle Y'(s),C(s)\rangle\,ds.
\]
Since \(M_\lambda=\exp(\theta\widehat J)\) commutes with $\widehat J$, we have  \(Y(0)=Y(2\pi)=\widehat J\),
and the boundary term vanishes, while
\[
\langle Y',C\rangle=-\lambda\,\langle [A,Y],C\rangle
=\lambda\,\langle Y,[A,C]\rangle
=\lambda a\,\langle Y,J_0\rangle.
\]
Thus
\begin{equation}\label{eq:theta}
    \theta'(\lambda)=\lambda \int_0^{2\pi}\!a(s)\langle Y(s),J_0\rangle\,ds.
\end{equation}

\emph{Step 4.}
As $J_0,S,T$ spans $\mathfrak{su}(1,1)$, write 
\[Y(s)=y(s)J_0+y_S(s)S+y_T(s) T,\]
Since Ad is an isometry for $\langle \cdot,\cdot\rangle$, 
\[\langle Y(s),Y(s)\rangle= \langle \widehat J(\lambda),\widehat J(\lambda)\rangle=1, \]
By (\ref{eq:gener}), it implies
\[y^2(s)-y_S^2(s)-y_T^2(s)=4 \implies |y(s)|\geq 2. \]

\emph{Step 5.} On the other hand, denote $\tilde{U}=\tilde{U}_{\lambda}(s):=G_{\lambda} U_{\lambda}(2\pi,s)$, then
 \[Y(s)=i \tilde{U}^{-1}J\tilde{U},\]
Denote $e_1=(1,0)^T$, a direct computation using the form (\ref{eq:xform}) for $Y(s)$ shows
\[y(s)=4 \langle Y(s), J_0\rangle= 2\,\mathrm{Im}\langle e_1, Y(s)e_1\rangle_J.\]
Let $v=\tilde{U} e_1$, then 
\[\langle e_1, Y(s)e_1\rangle_J=\langle \tilde{U} e_1, \tilde{U} Y(s)e_1\rangle_J=i \langle v, Jv\rangle_J=i\|v\|^2.\]
therefore $y(s)=2\|v\|^2>0$, whence $y(s)\geq 2$.

By (\ref{eq:theta}), we have
\[\theta'(\lambda)=\lambda \int_0^{2\pi} a(s)y(s)\langle J_0,J_0\rangle ds\geq \lambda \mathcal{A}(\Gamma),\]
in any connected component of the elliptic set.
\end{proof}
\qed

\begin{rmk} The condition $\omega(\Gamma)=1$ is used to specify  positive orientation of the curve, under which condition the algebraic area $\A(\Ga)>0$. If one choose $\omega(\Gamma)=-1$, then the areal velocity $a(t)<0$ and $\A(\Ga)<0$. The conclusion  in Proposition \ref{main1} becomes
\[\theta'(\lambda)\leq \lambda \mathcal{A}(\Gamma),\quad \forall \lambda\in I_E.\]

\end{rmk}


\bigskip

\begin{thm} \label{main2}Suppose $\Gamma$ is a $C^2$ star-shaped closed curve with $\omega(\Gamma)=1$. Then
 there is only one connected component of the elliptic set, where
\be\label{eq:rate2}
    \rho'(R)\leq -\frac{\mathcal{A}(\Gamma)}{\pi R^3},\quad R>\underline{R}(\Gamma).
\ee
In particular, the critical B-length and turning B-length of $\Gamma$ coincide:
\[\underline{R}(\Gamma)=\overline{R}(\Gamma).\]

\end{thm}

\begin{proof} By Proposition \ref{main1}, on any connected component of the elliptic set of $\Gamma$, we have
    \[\rho'(R)=-\frac{\theta'(1/R)}{\pi R^2}\leq -\frac{\mathcal{A}(\Gamma)}{\pi R^3}.\]
    By definition, $R$ enters to an elliptic component right after $\mcr(\Gamma)$. Meanwhile, by Proposition \ref{asym}, $\rho(R)$ decreases to $0$ as $R$ goes to infinity, we conclude that $\rho$ could have only one elliptic component given by $(\overline{R}(\Gamma),\infty)$. Otherwise, there will be an $R>\mcr(\Gamma)$ such that $\rho(R)=0$, after that $\rho(R)$ has to strictly decrease and can not be positive again.

\end{proof}
\qed

Combine Theorem \ref{main2} and Proposition \ref{convex}, we get immediately 
\begin{cor} Suppose $\Gamma$ is a $C^2$ strictly convex simple closed curve. Then the corresponding monodromy map $\mathcal{M}_R$ is hyperbolic when $R<\underline{R}(\Gamma)$ and elliptic when $R>\underline{R}(\Gamma)$.

\end{cor}

In order to show that the above property is also true for general star-shaped curves. We need that there is no parabolic point inside the plateau of the rotation number function $\rho$.

\begin{prop} \label{nopara} Suppose $\Gamma$ is a $C^2$ star-shaped closed curve with $\omega(\Gamma)=1$. Then at the first plateau $(0,\mcr(\Gamma))$, all points are hyperbolic, i.e. there is no point $R\in (0,\mcr(\Gamma))$ such that $|\mathrm{tr}(M_R)|=2$.
    
\end{prop}

\begin{proof} We fix a star-shaped closed curve with parametrization $\Gamma(t)=r(t)(\cos t,\sin t)$. Without loss of generality, suppose $\min_{t\in [2\pi-\varepsilon_0,2\pi]}|r(t)|=\delta>0$ for some $\varepsilon_0<\delta/2$. 

\emph{Step 1.}
Fix a smooth nontrivial positive $2\pi$-periodic function $\phi$ supported in $(-1,0)$, having smooth contact with value $0$ at the boundary, and $\|\phi\|_{\infty}\leq 1$. Consider $\epsilon$ small enough such that $|\varepsilon|<\varepsilon_0$. Set $\phi_{\varepsilon}=\varepsilon\varphi(t/|\varepsilon|)$
and define the perturbed curve
\[
\Gamma_\varepsilon(t):=\big(r(t)+\phi_\varepsilon(t)\big)(\cos t,\sin t).
\]
Then $\Gamma_{\varepsilon}$ is star-shaped. For $\varepsilon>0$ this adds a tiny bump outward at the arc $t\in (2\pi-|\varepsilon|,2\pi)$; for $\varepsilon<0$ it indents inward.

Denote by $M_{R,\varepsilon}(t_2,t_1)$ the monodromy from $t_1$ to $t_2$ along $\Gamma_\varepsilon$,
and  $M_{R,\varepsilon}= M_{R,\varepsilon}(2\pi,0)$.
A direct concatenation shows
\begin{equation}\label{eq:conta}
    M_{R,\varepsilon}
=M_{R,\varepsilon}(2\pi,2\pi-|\varepsilon|)M_{R,0}(2\pi-|\varepsilon|,0)
= N_{\varepsilon,R} \,M_{R,0},
\end{equation}
where $N_{\varepsilon,R}:=M_{R,\varepsilon}(2\pi,2\pi-|\varepsilon|)\,M_{R,0}(2\pi-|\varepsilon|,2\pi)$
is the monodromy of the “short spliced” closed curve $C_\varepsilon$ that first runs
$\Gamma$ backwards and then runs $\Gamma_\varepsilon$ forward
on $[2\pi-|\varepsilon|,2\pi]$.

\medskip
\emph{Step 2.}
Let $N_{\varepsilon,R}(t)$ satisfy $N'(t)=-(2R)^{-1}B_{\varepsilon}(t)\,N(t)$ with $N(0)=I$ along $C_\varepsilon$. 
Denote $N_{\varepsilon,R}=N_{\varepsilon,R}(2|\varepsilon|)$. Since the tangent vector of the curve $C_{\epsilon}$ has bounded norm, for $\varepsilon$ sufficiently small and $R$ fixed, by Lemma \ref{magnus-area},
\[N_{\varepsilon}=N_{\varepsilon,R}=\exp\!\big(h_R(\varepsilon)J_0+O(|\varepsilon|^3/R^3)).\]
where $h_{R}(\varepsilon)= \frac {1}{R^2}\mathcal{A}(C_{\varepsilon})$ is proportional to the signed area enclosed by $C_\varepsilon$. Hence $h(\varepsilon)$ has opposite signs for $\varepsilon>0$ and $\varepsilon<0$, and
$|h_R(\varepsilon)|\asymp \varepsilon^2$.


Using the expansion $\mathrm{tr}( e^X )=2+1/2\,\mathrm{tr}(X^2)+O(\|X\|^4)$, we get that $N_{\varepsilon,R}$ is elliptic for $\varepsilon\neq 0$ small.
Moreover, there exists $P_\varepsilon: = P_{\varepsilon,R}\in{\rm SU}(1,1)$ with $\|P_\varepsilon-I\|\to 0$ and
a continuous angle function $\theta(\varepsilon):=\theta_R(\varepsilon)=h_R(\varepsilon)+O(\varepsilon^3)$ such that
\begin{equation}\label{eq:Neps-rot}
N_\varepsilon=P_\varepsilon^{-1}\begin{pmatrix}\cos\theta(\varepsilon)+i\sin\theta(\varepsilon)&0\\
0&\cos\theta(\varepsilon)-i\sin\theta(\varepsilon)\end{pmatrix}P_\varepsilon.
\end{equation}

\medskip
\emph{Step 3.}
Assume, for contradiction, that there exists $R_1\in(0,R(\Gamma))$ with $|\text{tr}(M_{R_1,0})| = 2$ and some $R_2\in(R_1,R(\Gamma))$ with $M_{R_2,0}$ hyperbolic.
 Let us suppose $\tr(M_{R_1,0})=-2$ (the case $\tr(M_{R_1,0})=2$ is similar). If $M_{R_1,0}=-I$. Then
$M_{R_1,\varepsilon}=-\,N_{\varepsilon,R_1}$ is elliptic for all small $\varepsilon\ne 0$.  If $M_{R_1,0}$ is  parabolic, write
\[
P_{\varepsilon}M_{R_1,0}\,P_{\varepsilon}^{-1}=
\begin{pmatrix}-1+ic(\varepsilon)&*\\ *&-1-ic(\varepsilon)\end{pmatrix},\qquad c(0)\ne 0.
\]
Using \eqref{eq:conta} and \eqref{eq:Neps-rot} and taking traces,
\[
\mathrm{tr}(\,M_{R_1,\varepsilon})=\mathrm{tr}\big(\mathrm{diag}\big(e^{i\theta(\varepsilon)},e^{-i\theta(\varepsilon)}\big)P_\varepsilon M_{R_1,0}P_\varepsilon^{-1}\big)
=-2\cos\theta(\varepsilon)-2c(\varepsilon)\sin\theta(\varepsilon).
\]
Since $\theta(\varepsilon)=h(\varepsilon)+O(\varepsilon^3)$ and $c(\varepsilon)$ has fixed sign for
$|\varepsilon|$ small, we can pick the sign of $\varepsilon$ so that
$\sin\theta(\varepsilon)$ and $c(\varepsilon)$ have the opposite signs. Then
$\mathrm{tr}\,M_{R_1,\varepsilon}>-2$, i.e. $M_{R_1,\varepsilon}$ is elliptic.

Meanwhile, using $\tr(M_{R_2,\varepsilon})=\tr(N_{\varepsilon,R_2} M_{R_2,0})$, where $\|N_{\varepsilon,R_2} - I\|\to 0$, we can take $\varepsilon$ small enough such that $|\tr(M_{R_2,\varepsilon})|>2$. Thus $\Gamma_{\varepsilon}$ has two connected components of elliptic set, contradicts with Theorem \ref{main2}.
\end{proof}
\qed
\bigskip

By Theorem \ref{main2} and Proposition \ref{nopara}, we have

\begin{thm}\label{trans} Suppose $\Gamma$ is a $C^2$ star-shaped closed curve. Then the corresponding monodromy map $\mathcal{M}_R$ is hyperbolic when $R<\underline{R}(\Gamma)$ and is elliptic when $R>\underline{R}(\Gamma)$.
    
\end{thm}

 \appendix
 \appendixpage           
\addappheadtotoc

\section{Magnus series}\label{sec:magnus}
The Magnus expansion provides a canonical way to write the fundamental solution of a linear matrix ODE
\[U'(t)=A(t)U(t), \qquad U(0)=I\]
 as a single exponential \(U(t)=\exp\big(\Omega(t)\big)\).
Here \(\Omega(t)=\sum_{k\ge1}\Omega_k(t)\) is a Lie series built from time–ordered integrals of nested commutators of \(A\). See \cite{M54, BCOR09} for details. 

In the following , the notation $O(\cdot)$ stands for a matrix whose norm is  bounded by a constant of the argument as the argument goes to $0$. 

Set
\[
g(t):=\int_0^t \|A(s)\|\,ds.
\]

\begin{thm}[Two-term Magnus exponent and $O(g^3)$ remainder]
\label{thm:magnus-2}
Fix $t\in[0,L]$ and assume $g(t)$ is sufficiently small. Then the logarithm $\Omega(t):=\log U(t)$ is well-defined and admits the expansion
\[
\Omega(t)\;=\;\Omega_1(t)+\Omega_2(t)+O(g(t)^3),
\]
where
\[
\Omega_1(t)=\int_0^t A(s)\,ds,\qquad
\Omega_2(t)=\frac12\!\!\int_{0<s_2<s_1<t}\!\![A(s_1),A(s_2)]\,ds_2ds_1.
\]
\end{thm}

\begin{proof}
By Picard iteration,
\[
U(t)=I+\sum_{k\geq 1}U_k(t):=I+\sum_{k\ge1} \int_{0<s_k<\cdots<s_1<t} A(s_1)\cdots A(s_k)\,ds_k\cdots ds_1.
\]
Submultiplicativity and the volume of the ordered simplex give
\[
\Big\|\int_{0<s_k<\cdots<s_1<t} \prod_{j=1}^k A(s_j)\,ds\Big\|\le g(t)^k/k!
\]
Hence
\[
U(t)=I+U_1(t)+U_2(t)+O\!\big(g(t)^3\big).
\]

Write $U(t)=I+X(t)$ with $X(t)=U_1+U_2+O(g^3)$ and $\|X(t)\|\le e^{g(t)}-1=O(g(t))$.
For $\|X\|$ small,
\[
\log(I+X)=X-\tfrac12 X^2+O\!\big(\|X\|^3\big)
=X-\tfrac12 X^2+O\!\big(g(t)^3\big).
\]
Insert $X=U_1+U_2+O(g^3)$; since $U_2=O(g^2)$, the cross terms $U_1U_2$ and $U_2U_1$ are $O(g^3)$, and $U_2^2=O(g^4)$. Thus
\[
\Omega(t)=\log U(t)
= \underbrace{U_1(t)}_{\Omega_1(t)}
+\underbrace{\big(U_2(t)-\tfrac12 U_1(t)^2\big)}_{\Omega_2(t)}
+O\!\big(g(t)^3\big).
\]

Expand
\[
\tfrac12 U_1(t)^2
=\tfrac12\!\int_0^t\!\!\int_0^t A(s_1)A(s_2)\,ds_2ds_1
=\tfrac12\!\int_{s_2<s_1} \!\!A(s_1)A(s_2)\,ds_2ds_1
+\tfrac12\!\int_{s_2<s_1}\!\! A(s_2)A(s_1)\,ds_2ds_1.
\]
Therefore,
\[
U_2(t)-\tfrac12 U_1(t)^2
=\tfrac12\!\int_{s_2<s_1}\!\!\big(A(s_1)A(s_2)-A(s_2)A(s_1)\big)\,ds_2ds_1
=\tfrac12\!\!\int_{0<s_2<s_1<t}\!\![A(s_1),A(s_2)]\,ds_2ds_1,
\]
which is the desired commutator formula. This completes the proof.
\end{proof}

\begin{rmk} Let $C=\max_{s\in [0,L]}\|A(s)\|$. Then $g(t)\le Ct$. We can always replace $g(t)$ by $Ct$ in the $O(\cdot)$ notation.
\end{rmk}

\bigskip

{\scshape Diantong Li}, {\scshape School of Mathematical Sciences, University of Chinese Academy of Sciences, Beijing,  China}\\
\indent \textit{Email address:} \texttt{lidiantong25@mails.ucas.ac.cn}

\medskip
{\scshape Qiaoling Wei, School of Mathematical Sciences, Capital Normal University, Beijing, China}\\
\indent \textit{Email address:} \texttt{wql03@cnu.edu.cn}

\medskip
{\scshape Meirong Zhang, Department of Mathematical Sciences, Tsinghua University, Beijing, China}\\
\indent \textit{Email address:} \texttt{zhangmr@tsinghua.edu.cn}

\medskip
{\scshape Zhe Zhou, State Key Laboratory of Mathematical Sciences, Academy of Mathematics and Systems Science, Chinese Academy of Sciences, Beijing, China.}\\
\indent \textit{Email address:} \texttt{zzhou@amss.ac.cn}

\end{document}